\documentclass{article}

\usepackage[
	a4paper,
	margin=1in
]{geometry}
\usepackage{setspace}
\usepackage{sectsty} 
\usepackage{titlesec}
\usepackage{appendix}

\usepackage{amsmath, amssymb, amsthm}
\usepackage{mathtools}

\usepackage[
	setpagesize=false,
	colorlinks=true,
	linkcolor=BrickRed,
        citecolor=OliveGreen,
        urlcolor=black,
	pdfencoding=auto,
	psdextra,
]{hyperref}
\usepackage{cleveref}

\usepackage{etoolbox} 
\usepackage[shortlabels]{enumitem}
\usepackage{xparse} 

\usepackage[dvipsnames]{xcolor}
\usepackage{graphicx, color}
\usepackage{epsfig}
\usepackage{tikz}
\usetikzlibrary{calc,decorations.pathmorphing,decorations.text}
\usepackage{subcaption}
\usepackage{refcount}

\usepackage{todonotes}
\usepackage[
    deletedmarkup=sout,
    commentmarkup=uwave,
    authormarkuptext=name
]{changes}

\usepackage{comment}

\setlist[enumerate,1]{label=(\arabic*), ref=(\arabic*)}
\setlist[enumerate,3]{label=(\roman*), ref=(\roman*)}

\allsectionsfont{\boldmath} 

\titlelabel{\thetitle.\quad}

\theoremstyle{plain}
\newtheorem{theorem}{Theorem}[section]
\newtheorem{lemma}[theorem]{Lemma}

\newtheorem{conjecture}[theorem]{Conjecture}

\newtheorem{claim}[theorem]{Claim}
\newtheorem*{claim*}{Claim}

\makeatletter
\newenvironment{claimproof}[1][Proof]{\par
	\pushQED{\qed}%
	
	\normalfont \topsep6\p@\@plus6\p@\relax
	\trivlist
	\item[\hskip\labelsep
	\textit{#1}\@addpunct{.}~]\ignorespaces
}{%
	\popQED\endtrivlist\@endpefalse
}
\makeatother

\newlist{Cases}{enumerate}{3}
\setlist[Cases]{parsep=0pt plus 1pt}
\setlist[Cases,1]{wide=0pt, listparindent=\parindent,
    label = \textbf{Case~\arabic*:}, ref = \arabic*}
\setlist[Cases,2]{wide=\parindent, listparindent=\parindent,
    label = \textbf{Case~\arabic{Casesi}-\arabic{Casesii}:}}

\crefname{Casesi}{case}{cases}

\newcounter{case}
\AtBeginEnvironment{proof}{\setcounter{case}{0}}

\crefname{case}{case}{cases}

\theoremstyle{definition}

\newcommand{\RR}{\mathbb{R}}

\title{Note on Long Directed Cycles in Eulerian Digraphs}
\author{Jiangdong Ai\thanks{School of Mathematical Sciences and LPMC, Nankai University. {\tt jd@nankai.edu.cn}.}
\hspace{2mm}
Gregory Gutin\thanks{Department of Computer Science, Royal Holloway University of London, {\tt g.gutin@rhul.ac.uk}, and School of Mathematical Sciences and LPMC, Nankai University.}
\hspace{2mm} Fankang He\thanks{School of Mathematical Sciences and LPMC, Nankai University. {\tt  
hefankang@mail.nankai.edu.cn}.}
\hspace{2mm} Anders Yeo\thanks{Department of Mathematics and Computer Science, University of Southern Denmark. {\tt andersyeo@gmail.com}, and Department of Mathematics, University of Johannesburg.}
}
\date{}

\begin{document}

\maketitle

\begin{abstract}
Huang, Ma, Shapira, Sudakov and Yuster (Comb. Prob. Comput. 2013) proved that every Eulerian digraph of average out-degree $d$ has a directed cycle of length at least $\sqrt{d}.$ We improve the lower bound from $\sqrt{d}$ to $\sqrt{2d}-3/2.$ 
\end{abstract}

\section{Introduction}
In this paper, digraphs do not have loops or multiple arcs. 
A digraph is {\em Eulerian} if it is strongly connected and the out-degree of every vertex is equal to its in-degree. 
Bollob\'as and Scott~\cite{Bollobas1996weighted} stated the following conjecture.

\begin{conjecture}[\cite{Bollobas1996weighted}]
    Every $n$-vertex Eulerian digraph $G$ with a weight function $w: A(G) \to \RR_{\geq 0}$ has a directed cycle of length at least $c\,\cdot \, w(G)/n$ for some universal constant $c>0$.
\end{conjecture}

Huang, Ma, Shapira, Sudakov and Yuster~\cite{Huang2013feedback} suggested trying to resolve the following two weaker versions of Bollob\'as-Scott conjecture as the first step in its resolution. 

\begin{conjecture}
    Every Eulerian digraph with average out-degree $d$ has a directed cycle of length at least $cd$ for some universal constant $c>0$.
\end{conjecture}

\begin{conjecture}
    Every Eulerian digraph with average out-degree $d$ has a directed path of length at least $cd$ for some universal constant $c>0$. 
\end{conjecture}

The authors of~\cite{Huang2013feedback} proved that every $n$-vertex Eulerian digraph with average out-degree $d$ has a directed cycle of length at least $\sqrt{d}$, and hence it has a directed path of length $\sqrt{d}-1$. 
Janzer, Sudakov and Tomon~\cite{Janzer2021Eulerian} first broke the $\sqrt{d}$ barrier for paths by showing the existence of a directed path of length at least $d^{1/2 + \epsilon}$, where $\epsilon = 1/40$. 
In fact, they proved a slightly stronger result that every vertex can start such a directed path. 
Subsequently, Knierim, Larcher and Martinsson~\cite{Knierim2021paths} gave an almost linear bound by showing that there is a directed path of length at least $d/(\log d + 1)$ in such digraphs.

However, for the existence of long cycles, $\sqrt{d}$ obtained in $\cite{Huang2013feedback}$ is still the best-known lower bound. 
If we take maximum degree into account, Knierim, Larcher, Martinsson and Noever~\cite{Knierim2021cycles} proved that an Eulerian digraph $D$ with average out-degree $d$ and maximum degree $\Delta$ has a directed cycle of length at least $\Omega(d/\log \Delta)$ by using a random walk technique. 
Moreover, they remarked that it seems likely that an $\Omega(d/\log d)$ bound can be obtained by using their approach more carefully. 
However, no existing paper does this.

The authors of~\cite{Huang2013feedback} recall the argument using the DFS tree to find long cycles in undirected graphs and remarked that simply using such an argument in digraph settings will fail. They also pointed out that a possible reason that this approach fails is the fact that not all non-tree arcs are \emph{back arcs}, that is, arcs directed from vertices to their ancestors. 
In this note, we use a \emph{final out-branching} instead of a DFS tree to give the following result which improves the multiplicity constant of the result of Huang, Ma, Shapira, Sudakov and Yuster~\cite{Huang2013feedback}. 

\begin{theorem}\label{Thm:main}
    Let $D$ be an $n$-vertex Eulerian digraph with average out-degree $d$. Then there is a directed cycle of length at least $\sqrt{2d}-3/2$.
\end{theorem}

\section{Proofs}
An {\em out-tree} $T$ in a digraph $D$ is an oriented tree in which every vertex, apart from one called the {\em root} of $T$, has in-degree 1. An {\em out-forest} is a disjoint union of out-trees and an {\em out-branching} of $D$ is a spanning out-tree of $D$. 

Let $D$ be a Eulerian digraph. Since $D$ is strongly connected, it has an out-branching $F_0$ (see, e.g., Proposition 1.8.1 in \cite{BangJensenGutin2018}). 
For each out-branching of $D$ rooted at $v_0$, we say a vertex $u$ is in \emph{level $i$}, $L_i$, if the unique directed path in $F_0$ from $v_0$ to $u$ is of length $i$. 
If $L_i$ and $L_j$ are two levels and $i<j$ then we say that
$L_i$ is a \emph{lower level} and $L_j$ is an \emph{upper level}. 

The notion of a \emph{final spanning out-forest} was explicitly introduced by El Sahili and Kouider~\cite{Sahili2007blocks} to solve a special case of Burr's conjecture on oriented paths with two blocks. (The term final spanning out-forest was coined in \cite{ABHT2007}.)
Let us recall the notion to make this note self-contained. For the same reason, we give proofs of the following lemmas. 
Let $F$ be an out-branching of $D$. Suppose $D$ has an arc $uv$ that is not a back arc, where $u$ lies at level $i$, $v$ lies at level $j$ and $i \geq j$. 
Then we can obtain a new out-branching $F'$ of $D$ by adding arc $u$ removing the arc of $F$ with head $v$. 
We call above operation an \emph{elementary operation}. 
Note that we can iteratively apply elementary operations on out-branchings of $D$ until no such arcs exist, since each time we apply elementary operation there is a vertex going to an upper level. 
An out-branching is a \emph{final out-branching} if we can not apply the elementary operation on it. 

By the definition of a final out-branching, we have the following lemma guaranteeing that all arcs sending from upper levels to lower levels are back arcs. Recall that an arc $uv$ in a digraph $D$ with out-branching $F$ is a {\em back arc} if $u$ is in an upper level than $v$ and there is a directed path in $F$ from $u$ to $v.$ 
\begin{lemma}\label{Lem:back_arc}
    Let $F$ be a final out-branching of $D$. Then all arcs directed from upper levels to lower levels are back arcs. Moreover, each level is an independent set. 
\end{lemma}

Given a digraph $D$ and $W \subseteq V(D)$, we use $d^+(W)$ to denote the number of arcs from $W$ to $V(D)\setminus W$ and $d^-(W)$ to denote the number of arcs from $V(D)\setminus W$ to $W$. The following lemma is well known, see, e.g., Proposition 4.1.1 in \cite{BangJensenGutin2018}.
\begin{lemma}\label{Lem:Euler_cut}
    Let $D$ be an Eulerian digraph and $W \subseteq V(D)$, then $d^+(W) = d^-(W)$.  
\end{lemma}
\begin{proof}
    Note that 
    \[
    d^+(W) = \sum_{v \in W} d^+(v) - |A(D[W])|,
    \]
    and 
    \[
    d^-(W) = \sum_{v \in W} d^-(v) - |A(D[W])|.
    \]
    Since $D$ is Eulerian, we have $d^+(x) = d^-(x)$ for every $x \in V(D)$. This yields $d^+(W) = d^-(W)$. 
\end{proof}

\begin{proof}[Proof of Theorem~\ref{Thm:main}]
    Suppose the longest directed cycle in $D$ is of length $t$. Take $F$ as a final out-branching of $D$, then the number of back arcs is at most $n(t-1)$, since each vertex can send at most $t-1$ arcs to its ancestors 
    otherwise there will be a directed cycle of length at least $t+1$ which contradicts the fact that the circumference of $D$ is $t$. 
    We call the arcs of $D$ from lower levels of $F$ to upper levels of $F$ \emph{forward arcs}. Note that the arcs in $F$ are also forward arcs. 
    \begin{claim}
        The number of forward arcs is at most $nt(t+1)/2$.
    \end{claim}
    \begin{claimproof}
        Let $L_i$ be the set of vertices at  $i$, and $L_{\leq i} = \bigcup_{0\leq j \leq i} L_j$. 
        By Lemma~\ref{Lem:Euler_cut} and Lemma~\ref{Lem:back_arc}, 
        \[
        d^+(L_{\leq i}) = d^-(L_{\leq i}) \leq \sum_{j=1}^t (t+1-j)|L_{i+j}|.
        \]
        Then the number of forward arcs is at most
        \[
        \sum_{i=0}^\infty d^+(L_{\leq i}) \leq \sum_{i=0}^\infty \sum_{j=1}^t (t+1-j)|L_{i+j}| \leq \sum_{i=0}^\infty \sum_{j=1}^t j |L_i| = \frac{t(t+1)}{2} \sum_{i=0}^\infty |L_i| \leq \frac{t(t+1)}{2} n.  \qedhere
        \]
    \end{claimproof}
    
    Hence,
    \[
    nd = |A(D)| \leq  n(t-1) + nt(t+1)/2,
    \]
    which implies that $t \geq \sqrt{2d}-3/2$. 
\end{proof}

\section{Concluding remarks}
By a similar argument, we can prove the following weaker version of the result of Janzer, Sudakov and Tomon~\cite{Janzer2021Eulerian} mentioned in Introduction. 

\begin{theorem}
    Let $D$ be an $n$-vertex Eulerian digraph with average out-degree $d$. Then for any vertex $v \in V(D)$, there exists a directed path of length at least $\sqrt{2d}-3/2$ starting at $v$.
\end{theorem}

\bibliographystyle{abbrv}
\bibliography{bibfile}
\label{key} 

\end{document}